\documentclass{article}
\usepackage{amssymb}
\usepackage{amsfonts}
\usepackage{amsmath}
\usepackage[doublespacing]{setspace}

\newtheorem{theorem}{Theorem}

\newtheorem{corollary}[theorem]{Corollary}

\newtheorem{definition}[theorem]{Definition}

\newtheorem{lemma}[theorem]{Lemma}

\newtheorem{proposition}[theorem]{Proposition}
\newtheorem{remark}[theorem]{Remark}

\newenvironment{proof}[1][Proof]{\noindent\textbf{#1.} }{\ \rule{0.5em}{0.5em}}
\input{tcilatex}
\begin{document}

\title{Intersection Theorems for Closed Convex Sets and Applications}
\author{Hichem Ben-El-Mechaiekh \\
Department of Mathematics, Brock University, \smallskip \\
Saint Catharines, Ontario, Canada, L2S 3A1}
\date{ \textit{To Appear in Missouri J. Math. Sc.}\\
\textit{(Sumitted: November 9, 2011; Accepted: May 22, 2012)}}
\maketitle

\begin{abstract}
A number of landmark existence theorems of nonlinear functional analysis
follow in a simple and direct way from the basic separation of convex closed
sets in finite dimension via elementary versions of the
Knaster-Kuratowski-Mazurkiewicz principle - which we extend to arbitrary
topological vector spaces - and a coincidence property for so-called von
Neumann relations. The method avoids the use of deeper results of
topological essence such as the Brouwer fixed point theorem or the Sperner's
lemma and underlines the crucial role played by convexity.\medskip\ It turns
out that the convex KKM principle is equivalent to the Hahn-Banach theorem,
the Markov-Kakutani fixed point theorem, and the Sion-von Neumann minimax
principle.

\textit{Keywords and phrases:} Separation of convex sets, intersection
theorems, convex KKM theorem, fixed points for von Neumann relations,
coincidence, systems of nonlinear inequalities, variational inequalities,
minimization of functionals, Markov-Kakutani fixed point theorem,
Hahn-Banach theorem.

\textit{2010 AMS\ Subject Classification: }Primary: 52A07, 32F32, 32F27,
Secondary: 47H04, 47H10, 47N10
\end{abstract}

\section{Introduction}

The aim of this expository paper is to show that a number of landmark
results of nonlinear functional analysis can be quickly obtained from a
particular version of the KKM principle at little cost. This "elementary
KKM\ principle" is due to A. Granas and M. Lassonde in the framework of
super-reflexive Banach spaces [10]. It is extended to arbitrary topological
vector spaces, under a more general compactness hypothesis, with a simpler
proof based on the separation of closed convex subsets in a Euclidean space
(a result usually discussed in a first course of continuous optimization)
and an intersection theorem of V. L. Klee [17]. A similar approach is
followed to formulate a coincidence theorem for so-called \textit{von
Neumann relations}.

The methods outlined here allow for a shorter and simpler alternative
treatment of existence results of functional analysis that avoids involved
and deeper principles that require sophistication and investment in time.
The KKM\ principle is a striking example of such fundamental results.
Indeed, using the Sperner lemma as a starting point, three of the greatest
topologists of all times, Polish academician S. Mazurkiewicz and two of his
former\ doctoral students, B. Knaster and K. Kuratowski published in 1929
the celebrated KKM\ lemma: a remarkable intersection theorem for closed
covers of a Euclidean simplex [18]. They used the KKM lemma to provide a
combinatorial proof of the Brouwer fixed point theorem (the two results
being in fact equivalent). In 1961, Ky Fan extended the KKM\ Lemma to vector
spaces of arbitrary dimensions in what became known as the KKM\ principle
[12]. The KKM\ principle inspired countless mathematicians, yielding a
formidable body of work in nonlinear and convex analysis; a production known
today as the KKM\ theory. The reader is referred to Dugundji-Granas [7],
Park [20] and Yuan [22] for surveys of results, methods, and applications of
the KKM\ theory.

The particular version of the KKM\ principle discussed here, which we call
the \textit{convex KKM\ principle}, is more than sufficient to prove in a
direct and economical way, such fundamental results as the Stampacchia
Theorem on variational inequalities, the Mazur-Schauder theorem on the
minimization of lower semicontinuous quasiconvex and coercive functionals,
and the Markov-Kakutani fixed point theorem for commuting families of affine
transformations (see e.g., Br\'{e}zis [6]). It is well-known, since Kakutani
[16], that the Hahn-Banach theorem can be derived from the Markov-Kakutani
fixed point theorem. Thus, the equivalence between the Hahn-Banach theorem,
Klee's intersection theorem, the convex KKM\ principle, and the
Markov-Kakutani fixed point theorem is thus established.

\section{Preliminaries}

The fundamental tool for our proof of the convex KKM theorem is the
separation of a point and a closed convex set in a finite dimensional space.
For the sake of completeness, we\ include the basic separation properties in
finite dimensions with the simplest of proofs (see e.g., Magill and Quinzii
[19])\bigskip

\begin{lemma}
\textit{Let }$C$\textit{\ be a nonempty closed convex subset of }$%
\mathbb{R}
^{n}$\textit{\ and let }$x\notin C.$\textit{\ Denote by }$y=P_{C}(x)$\textit{%
\ the projection of }$x$\textit{\ onto }$C$\textit{. Then the hyperplane }$%
H_{C}^{x}$\textit{\ orthogonal to }$u=x-y$\textit{\ passing through }$y$ 
\textit{strictly separates }$x$\textit{\ and }$C$\textit{, namely:} 
\begin{equation*}
\langle u,z\rangle \leq \langle u,y\rangle <\langle u,x\rangle ,\forall z\in
C
\end{equation*}
\end{lemma}

\begin{proof}
Since $C$ is closed and convex, the projection $y=P_{C}(x)$ of $x$ onto $C$
is unique. Define, for any given $z\in C,$ a functional $\varphi
_{z}:[0,1]\longrightarrow 
\mathbb{R}
$ by:%
\begin{equation*}
\varphi _{z}(t):=\Vert x-(tz+(1-t)y\Vert ^{2}.
\end{equation*}%
As $y$ is closest to $x,$ $\varphi _{z}(t)$ achieves its minimum on [0, 1]
at $t=0$, thus $\varphi _{z}^{\prime }(0)\geq 0$. Since $\varphi
_{z}^{\prime }(t)=2t\Vert y-z\Vert ^{2}+2\langle x-y,y-z\rangle ,$ it
follows $\varphi _{z}^{\prime }(0)=2\langle x-y,y-z\rangle =2\langle
u,y-z\rangle \geq 0,$ i.e., $\langle u,z\rangle \leq \langle u,y\rangle .$
On the other hand, $0<\Vert x-y\Vert ^{2}=\langle u,x-y\rangle =\langle
u,x\rangle -\langle u,y\rangle .$ Thus, $\langle u,z\rangle \leq \langle
u,y\rangle <\langle u,x\rangle .$
\end{proof}

\begin{proposition}
\textit{Let }$K$\textit{\ and }$C$\textit{\ be disjoint convex subsets of }$%
\mathbb{R}
^{n}$\textit{\ with }$K$\textit{\ compact and }$C$\textit{\ closed. Then, }$%
C $\textit{\ and }$K$\textit{\ are strictly separated by a hyperplane }$H,$%
\textit{\ i.e., }%
\begin{equation*}
\mathit{there\ exists\ }u\in 
\mathbb{R}
^{n},u\neq 0,\mathit{\ with}\text{ }\sup_{x\in C}\langle u,x\rangle
<\min_{x^{\prime }\in K}\langle u,x^{\prime }\rangle .
\end{equation*}%
\bigskip
\end{proposition}

\begin{proof}
Since $C$ is closed and $K$ is compact, the set $C-K:=\{y\in 
\mathbb{R}
^{n}:y=x-x^{\prime },x\in C,x^{\prime }\in K\}$ is also closed. It is,
moreover, convex as the difference of convex sets. Since $C\cap K=\emptyset
, $ then $0\notin C-K.$ Lemma 1 applies, yielding for $u=0-P_{C-K}(0),$ the
inequalities: $\langle u,z\rangle \leq \langle u,-u\rangle <\langle
u,0\rangle =0,\forall z\in C-K.$ Thus, as $z=x-x^{\prime },x\in C,x^{\prime
}\in K,$%
\begin{equation*}
\langle u,x\rangle \leq \langle u,x^{\prime }\rangle -\Vert u\Vert
^{2}<\langle u,x^{\prime }\rangle ,\forall x\in C,\forall x^{\prime }\in K.
\end{equation*}%
\bigskip
\end{proof}

A refinement of a fundamental intersection theorem of V. L. Klee for
families of closed convex subsets of $%
\mathbb{R}
^{n\text{ }}$ (see Klee [17] and Berge [5]) plays a crucial role in our
proof. We provide here a simple proof based on Proposition 2.

Topological vector spaces (t.v.s. for short), as well as topological spaces,
are assumed to be Hausdorff (T$_{2}$). Vector spaces are assumed real (or
complex) and the convex hull of a subset $A$ of a vector space is denoted by 
$conv(A).$\bigskip

\begin{proposition}
([11]) \textit{Let }$C_{1},\ldots ,C_{n},$\textit{\ be non empty closed
convex sets in a t.v.s. }$E$\textit{\ such that:}

\textit{(i) }$C=\bigcup_{i=1}^{n}C_{i}$\textit{\ is convex, and}

\textit{(ii) each }$k$ of them, $1\leq k<n,$ have a common point.

\textit{Then }$\bigcap_{i=1}^{n}C_{i}\neq \emptyset .$
\end{proposition}

\begin{proof}
The proof goes along the lines \ of Klee's proof [17]. One may assume with
no loss of generality that the sets $C_{i},i=1,\ldots ,n,$ are compact
convex subsets of a finite dimensional space. Indeed, one could consider the
convex finite polytope $\hat{C}:=Conv(\{y_{j}:j=1,\ldots ,n\}),$ where, for
each $j=1,\cdots ,n,$ the points $y_{j}\in \bigcap_{i=1,i\neq j}^{n}C_{i}$
are provided by (ii), and define $\hat{C}_{i}:=C_{i}\cap \hat{C}.$ Clearly,
all the sets $\hat{C}_{1},\ldots ,\hat{C}_{n},\hat{C}=\bigcup_{i=1}^{n}\hat{C%
}_{i}$ are compact convex sets in a finite dimensional subspace of $E$ and $%
\bigcap_{i=1}^{n}C_{i}\neq \emptyset \Longleftrightarrow \bigcap_{i=1}^{n}%
\hat{C}_{i}\neq \emptyset .$

If $n=1,$ the thesis clearly holds. Assume, for a contradiction that for $%
n\geq 2,$ $\bigcap_{i=1}^{n}C_{i}=\emptyset $ and let us show that (i) must
fail if (ii) holds true. The proof is by induction on $n.$

If $n=2,$ (ii) asserts that both $C_{1}$ and $C_{2}$ are nonempty and, while
they are disjoint, their union $C=C_{1}\cup C_{2}$ cannot be convex and thus
(i) fails.

Suppose that for $n=k-1,$ it holds $(\bigcap_{i=1}^{k-1}C_{i}=\emptyset $
and $\bigcap_{i=1,i\neq j}^{k-1}C_{i}\neq \emptyset )\Longrightarrow
\bigcup_{i=1}^{k-1}C_{i}$ is not convex.

Let $n=k,$ and let $\{C_{i}\}_{i=1}^{k}$ be a collection of compact convex
sets such that $C_{k}\cap \bigcap_{i=1}^{k-1}C_{i}=\emptyset $ and $\forall
j=1,...,k,$ $\bigcap_{i=1,i\neq j}^{k}C_{i}\neq \emptyset .$ By Proposition
2, the disjoint compact convex sets $C_{k}$ and $\bigcap_{i=1}^{k-1}C_{i}$
can be strictly separated by a hyperplane $H.$ Putting, for each $i=1,\ldots
,k,$ $C_{i}^{\prime }:=H\cap C_{i},$ it follows that $C_{k}^{\prime }$ and $%
\bigcap_{i=1}^{k-1}C_{i}^{\prime }$ are empty. Moreover, for a given
arbitrarily chosen $j_{0}\in \{1,\ldots ,k-1\},$ let $y_{0}\in
\bigcap_{i=1,i\neq j_{0}}^{k}C_{i},$ thus $y_{0}\in C_{k},$ and let $%
y_{k}\in \bigcap_{i=1}^{k-1}C_{i}$ be arbitrarily chosen. Clearly, the
points $y_{0}$ and $y_{k}$ are both in the larger convex set $%
\bigcap_{i=1,i\neq j_{0}}^{k-1}C_{i}$ and are also strictly separated by $H.$
The intersection $\bar{z}$ of the line segment $[y_{0},y_{k}]$ with $H$
belongs to $\bigcap_{i=1,i\neq j_{0}}^{k-1}C_{i}\cap H.$ $j_{0}$ being
arbitrary, hypothesis (ii) is verified for the collection $\{C_{i}^{\prime
}\}_{i=1}^{k-1}$ and $\bigcap_{i=1}^{k-1}C_{i}^{\prime }=\emptyset .$ By the
induction hypothesis, $\bigcup_{i=1}^{k-1}C_{i}^{\prime
}=\bigcup_{i=1}^{k-1}(C_{i}\cap H)$ is not convex$.$ Since $H\cap
C_{k}=\emptyset ,$ it follows that $\bigcup_{i=1}^{k}(C_{i}\cap
H)=\bigcup_{i=1}^{k-1}(C_{i}\cap H)$ is not convex and the proof is
complete.\bigskip
\end{proof}

\begin{remark}
Proposition 3 is due to A. Ghouila-Houri [11] and slightly extends the
following result of V. L. Klee (see also C. Berge [6]):

\textbf{Klee's Theorem }[17]\textbf{: }\textit{Let }$C$\textit{\ and }$%
C_{1},\ldots ,C_{n}$\textit{\ be closed convex sets in a Euclidean space
satisfying: (i) }$C\subseteq \bigcup_{i=1}^{n}C_{i}$\textit{\ and (ii) }$%
C\cap \bigcap_{i=1,i\neq j}^{n}C_{i}\neq \emptyset $\textit{\ for any }$%
j=1,2,...,n.$\textit{\ Then }$C\cap \bigcap_{i=1}^{n}C_{i}\neq \emptyset .$
This can be restated: $(C\cap \bigcap_{i=1,i\neq j}^{n}C_{i}\neq \emptyset $
and $C\cap \bigcap_{i=1}^{n}C_{i}=\emptyset )\Longrightarrow C\nsubseteq
\bigcup_{i=1}^{n}C_{i}.$
\end{remark}

\section{The Convex KKM\ Theorem}

We use the following terminology of Dugundji-Granas (see [7]):\bigskip

\begin{definition}
\textit{Given an arbitrary subset }$X$\textit{\ be of a vector space }$E$%
\textit{, a set-valued map }$\Gamma :X\longrightarrow 2^{E}$\textit{\ is
said to be a KKM map if for every finite subset }$\{x_{1},\ldots
,x_{n}\}\subseteq X$\textit{\ it holds:}%
\begin{equation*}
conv(\{x_{1},\ldots ,x_{n}\})\subset \bigcup\nolimits_{i=1}^{n}\Gamma
(x_{i}).
\end{equation*}%
\bigskip
\end{definition}

\begin{theorem}
(Convex KKM Theorem) \textit{Let }$E$ be a t.v.s.,\textit{\ }$\emptyset \neq
X\subseteq Y\subseteq E$ with $Y$ convex. If\textit{\ }$\Gamma
:X\longrightarrow 2^{Y}$\textit{\ is a set-valued map verifying:}

\textit{(i) }$\Gamma $\textit{\ is a KKM map;}

\textit{(ii) all values of }$\Gamma $\textit{\ are non-empty, closed and
convex.}

\textit{Then, the family }$\{\Gamma (x)\}_{x\in X}$\textit{\ has the finite
intersection property.}

\textit{If in addition, there exists a non-empty subset }$X_{0}$ of $X$
contained in a convex compact subset $D$ of $Y$ such that $\bigcap_{x\in
X_{0}}\Gamma (x)$ \textit{is compact, then }$\bigcap_{x\in X}\Gamma (x)\neq
\varnothing .$\bigskip
\end{theorem}

\begin{proof}
We prove that Proposition 3 is equivalent to Theorem 6.

$(\Longrightarrow )$ Let\textit{\ }$\Gamma :X\longrightarrow 2^{Y}$\ be a
KKM map with closed convex values. We show by induction on $n$ that $%
conv(\{x_{1},\ldots ,x_{n}\})\cap \bigcap_{i=1}^{n}\Gamma (x_{i})\neq
\emptyset ,$ for any finite subset $\{x_{1},\ldots ,x_{n}\}$ of $X.$

When $n=1,$ $x_{1}=conv(\{x_{1}\})\subset \Gamma (x_{1}).$

Assume that the conclusion holds true for any set with $n=k$ elements, and
let $n=k+1.$ Put $C=conv(\{x_{1},\ldots ,x_{n}\})$ and $C_{i}=\Gamma
(x_{i})\cap C.$ Since $\Gamma $ is KKM, $C\subseteq \bigcup_{i=1}^{n}\Gamma
(x_{i})$ which implies $C=\bigcup_{i=1}^{n}(\Gamma (x_{i})\cap
C)=\bigcup_{i=1}^{n}C_{i}$, a convex set. By the induction hypothesis, for
each $i,$ we have $conv(\{x_{1},\ldots ,\hat{x}_{i},\ldots ,x_{n}\})\cap
\bigcap_{j=1,j\neq i}^{n}\Gamma (x_{j})\neq \emptyset .$ Proposition 4
implies that $\bigcap_{i=1}^{n}(\Gamma (x_{i})\cap C)\neq \emptyset ,$ i.e., 
$\bigcap_{i=1}^{n}\Gamma (x_{i})\neq \emptyset $.\bigskip

$(\Longleftarrow )$ Assume $C_{1},\ldots ,C_{n},C=\bigcup_{i=1}^{n}C$ are
closed convex sets in a topological vector space satisfying hypotheses (i)
and (ii) of Proposition 3 above.

For each $j,$ let $x_{j}\in $ $\bigcap\nolimits_{i=1,i\neq j}^{n}C_{i}$ and
consider $X=\{x_{j}\}_{j=1}^{n}.$ The set $C$ being convex, $%
conv(X)\subseteq C$ and for all $j,i$ with $j\neq i,$ $x_{j}\in C_{i},$
which implies that $A_{i}=conv(\{x_{j}\}_{j=1,j\neq i}^{n})$ $\subset C_{i}.$
Define $\Gamma :X\longrightarrow 2^{C}$ by $\Gamma (x_{i}):=C_{i}$ for each $%
i=1,...,n.$ The values of $\Gamma $ are clearly closed and convex. Also, $%
conv(X)\subseteq C$ $=\bigcup_{i=1}^{n}(C_{i}\cap C)=\bigcup_{i=1}^{n}\Gamma
(x_{i}),$ and for each $\{x_{i_{1}},\ldots ,x_{i_{k}}\}\subset X,$ we have $%
conv(\{x_{i_{1}},\ldots ,x_{i_{k}}\})\subset A_{i_{j}}\subset
C_{i_{j}}=\Gamma (x_{i_{j}})$ for some $j\neq 1,\ldots ,k.$ Hence $%
conv(\{x_{i_{1}},\ldots ,x_{i_{k}}\})\subset \bigcup_{j=1}^{k}\Gamma
(x_{i_{j}}),$ i.e., $\Gamma $ is a KKM map. By the convex KKM theorem, $%
\bigcap_{i=1}^{n}\Gamma (x_{i})\neq \emptyset ,$ thus $%
\bigcap_{i=1}^{n}C_{i}\neq \emptyset .$

Assuming for a moment that $\bigcap_{x\in X}\Gamma (x)$ is contained in a
compact subset $K$ of $Y,$ then the conclusion $\bigcap_{x\in X}\Gamma
(x)\neq \varnothing $ would follow at once from the characterization of
compactness in terms of families of closed subsets having the finite
intersection property.

Observe now that the restriction/compression map $\Gamma
_{0}:X_{0}\longrightarrow 2^{D}$ defined by $\Gamma _{0}(x):=\Gamma (x)\cap
D,x\in X_{0},$ has compact convex values and is also a KKM map. Indeed, for
any subset $\{x_{1},\ldots ,x_{n}\}\subseteq X_{0},$ $conv(\{x_{1},\ldots
,x_{n}\})\subset (\bigcup_{i=1}^{n}\Gamma (x_{i}))\cap
D=\bigcup_{i=1}^{n}\Gamma _{0}(x_{i}).$ Therefore, $\bigcap_{x\in
X_{0}}\Gamma (x)\supseteq \bigcap_{x\in X_{0}}\Gamma _{0}(x)\neq \varnothing
.$ The conclusion follows immediately from the fact that $\bigcap_{x\in
X}\Gamma (x)\subseteq \bigcap_{x\in X_{0}}\Gamma (x)$ is compact and
non-empty.
\end{proof}

\begin{remark}
(i) Theorem 6 is an extension to topological vector spaces of the elementary
KKM theorem of Granas-Lassonde, stated in the context of super-reflexive
Banach spaces [10].

(ii) Theorem 6 obviously follows from the KKM\ principle of Ky Fan [12]
where the values of $\Gamma $ are not assumed to be convex. The latter
requires, however, much more involved analytical or topological results.
Indeed, the Ky Fan KKM\ principle is equivalent to Sperner's lemma, to the
Brouwer fixed point theorem, and to the Browder-Ky Fan fixed point theorem
(see e.g., [1, 2, 3]).

(iii)\ In this generality, the compactness condition in the KKM\ principle
is due to Ky Fan [14]. It obviously extends the earlier compactness
conditions: $Y$ is also compact, or all values of $\Gamma $ are compact, or
a single value $\Gamma (x_{0})$ is compact, or $\bigcap_{i=1}^{n}\Gamma
(x_{i})$ is compact for some\ finite subset $\{x_{1},\ldots ,x_{n})$ of $X$%
.\bigskip
\end{remark}

Naturally, the convex KKM\ theorem can be expressed as an equivalent fixed
point property for what we call a \textit{von Neumann relation}. Given a
subset $A$ of a cartesian product of two sets $X\times Y,$ denote by $A(x)$
and $A^{-1}(y)$ the respective sections $\{y\in Y:(x,y)\in A\}$ and $\{x\in
X:(x,y)\in A\};$ denote by $A^{-1}$ the subset $\{(y,x):(x,y)\in A\}.$

\begin{definition}
A von Neumann relation is a subset $A$ of a cartesian product $X\times Y,$
where $X$ and $Y$ are subsets of topological vector spaces, satisfying:

(i) for every $x\in X,$ the section $A(x)$ is convex and non-empty;

(ii) for every $y\in Y,$ the section $A^{-1}(y)$ is open in $X$ and $%
X\setminus A^{-1}(y)$ is convex.

Denote by $\mathcal{N}(X,Y)$ the class of von Neumann relations in $X\times
Y $ and by $\mathcal{N}^{-1}(X,Y):=\{A:X\longrightarrow 2^{Y}:A^{-1}\in 
\mathcal{N}(Y,X)\}.$
\end{definition}

Note that von Neumann relations are particular cases of $F^{\ast }-$maps (%
\textit{applications de Ky Fan}) introduced in [3].

\begin{theorem}
(Fixed Point for $\mathcal{N}-$maps) Let\textit{\ }$E$ be a t.v.s.,\textit{\ 
}$\emptyset \neq Y\subseteq X\subseteq E$ with $X$ convex, and let\textit{\ }%
$A\in \mathcal{N}(X,Y).$ If there exist a compact subset $K$ of $X$ and a
compact convex subset $D$ of $Y$ such that for every $x\in X\setminus K,$ $%
A(x)\cap D\neq \emptyset $, then $A$ has a fixed point, i.e., $(\hat{x},\hat{%
x})\in A$ for some $\hat{x}\in X.$
\end{theorem}

\begin{proof}
Define $\Gamma :Y\longrightarrow 2^{X}$ as $\Gamma (y):=Y\setminus
A^{-1}(y),y\in Y.$ Clearly, $\Gamma $ has closed and convex values. Also,
obviously, $A(x)=Y\setminus \Gamma ^{-1}(x),$ for any $x\in X.$

One readily verifies that the compactness condition in Theorem 9 is
equivalent to the compactness condition in Theorem 6. Indeed, $(x\notin
K\Longrightarrow \exists y\in A(x)\cap D)\Longleftrightarrow ((\forall y\in
D,y\notin A(x))\Longrightarrow x\in K)\Longleftrightarrow (\bigcap_{y\in
D}\Gamma (y)\subseteq K).$ The intersection $\bigcap_{y\in D}\Gamma (y)$
being closed in $K$ which is compact, is also compact. For any subset $Y_{0}$
of $D,$ it also holds $\bigcap_{y\in Y_{0}}\Gamma (y)$ is a compact subset
of $K.$

The fact that all sections $A(x),x\in X,$ are non-empty rules out the thesis
of Theorem 6 (indeed, $(A(x)\neq \emptyset ,\forall x\in
X)\Longleftrightarrow $\textit{\ }$\bigcap_{y\in Y}\Gamma (y)=\varnothing ).$
Therefore $\Gamma $ cannot be a KKM map, i.e., there exist $\{y_{1},\ldots
,y_{n}\}\subseteq Y$ and $\hat{y}\in conv(\{y_{1},\ldots ,y_{n}\})$ with $%
\hat{y}\notin \bigcup_{i=1}^{n}\Gamma (y_{i}),$ which (by DeMorgan's law) is
equivalent to $\hat{y}\in $ $\bigcap_{i=1}^{n}A^{-1}(y_{i})%
\Longleftrightarrow \{y_{1},\ldots ,y_{n}\}\subseteq A(\hat{y}).$ Since $A(%
\hat{y})$ is convex, $\hat{y}\in A(\hat{y})\,\ $and the proof is complete.
\end{proof}

\begin{remark}
(i) The proof of Theorem 9 clearly establishes its equivalence with Theorem
6.

(ii) Theorem 9 is a particular instance of the Browder-Ky Fan fixed point
theorem (where the convexity of $X\setminus A^{-1}(y)$ in Definition 8 is
dispensed with; see e.g., [2, 3, 4]).

(iii) Note that if $X$ is compact, the compactness condition in Theorem 9 is
vacuously satisfied with $K=X.$ To the best of our knowledge, in this
generality and in the context of the Browder-Ky Fan fixed point theorem,
this condition was first introduced in [2, 3] in 1982. It builds on the
so-called Karamardian \textit{coercivity condition} for complementarity
problems (early seventies), taken up in 1977 by Allen in the context of
fixed point theorems for set-valued maps (case where $K=C$); see [2, 3, 9]
for references and details.
\end{remark}

We end this section with a coincidence theorem between $\mathcal{N}$ and $%
\mathcal{N}^{-1}$ maps with a direct proof based on Proposition 3. We shall
make use of a well-known selection property enjoyed by $F^{\ast }-$maps of
[2], thus by $\mathcal{N}-$maps (see [2, 3]).

\begin{lemma}
Let $A\in \mathcal{N}(X,Y)$ with $Y$ convex. For any compact subset $K$ of $%
X $, there exist a continuous (single-valued) mapping $s:K\longrightarrow Y$
and a convex compact finite polytope $P\subseteq Y$ such that $s(x)\in
A(x)\cap P$ for all $x\in K.$
\end{lemma}

\begin{proof}
Since $\forall x\in X,A(x)\neq \emptyset ,$ then $X=\bigcup_{y\in
Y}A^{-1}(y) $, a union of open subsets of $X.$ By compactness, $K\subseteq
\bigcup_{i=1}^{n}A^{-1}(y_{i})$ for some finite subset $\{y_{1},\ldots
,y_{n}\}\subset Y.$ Let $\{\lambda _{i}:K\longrightarrow \lbrack
0,1]\}_{i=1}^{n}$ be a continuous partition of unity subordinated to the
open cover $\{A^{-1}(y_{i})\}_{i=1}^{n},$ and define $s:K\longrightarrow
P=conv(\{y_{1},\ldots ,y_{n}\})\subset Y$\ by putting $s(x):=\sum_{i=1}^{n}%
\lambda _{i}(x)y_{i},x\in K.$ Clearly, for a given $x\in K,$ $\lambda
_{i}(x)\neq 0\Longrightarrow x\in A^{-1}(y_{i})\Longleftrightarrow y_{i}\in
A(x).$ The section $A(x)$ being convex, the convex combination $s(x)\in
A(x)\cap P.$
\end{proof}

\begin{theorem}
(Coincidence $(\mathcal{N},\mathcal{N}^{-1})$) Let $X$ and $Y$ be two
non-empty convex subsets in topological vector spaces, and let $A\in 
\mathcal{N}(X,Y),B\in \mathcal{N}^{-1}(X,Y).$

If one of the following compactness conditions holds: (i) $Y$ is compact; or
(ii) $X$ is compact; or (iii) there exist a compact subset $K$ of $X$ and a
compact convex subset $C$ of $Y$ with $A(x)\cap C\neq \emptyset ,\forall
x\in X\setminus K;$ then $A\cap B\neq \emptyset .$
\end{theorem}

\begin{proof}
Let it be made clear, first, that $(i)\Longrightarrow (iii)$ and $%
(ii)\Longrightarrow (iii).$ Indeed, if $Y$ is compact, take $C=Y$ and $%
K=\emptyset $ in $(iii).$ If $X$ is compact, take $K=X$ and $(iii)$ is
vacuously satisfied. Moreover, due to Lemma 11, $(iii)$ can be reduced to $%
(i)$. Indeed, assume that $(iii)$ holds. Lemma 11 implies the existence of a
convex finite polytope $P\subset Y$ and a continuous mapping $%
s:K\longrightarrow Y$ with $s(x)\in A(x)\cap P$ for all $x\in K.$ Now, if $%
x\in X\setminus K,A(x)\cap C\neq \emptyset $ where $C\subset Y$ is convex
and compact$.$ Consider the convex hull $\hat{Y}=conv(P\cup C)$, a compact
subset of $Y$ (the convex envelope of two compact convex sets in a
topological vector space is also compact convex). It is clear that the map $%
\hat{A}:X\longrightarrow 2^{\hat{Y}}$ given by $\hat{A}(x):=A(x)\cap \hat{Y}%
,x\in X,$ defines a von Neumann relation, i.e., $\hat{A}\in \mathcal{N}(X,%
\hat{Y}).$ Also, the mapping $\hat{B}:X\longrightarrow 2^{\hat{Y}}$ given by 
$\hat{B}(x):=B(x)\cap \hat{Y},x\in X,$ verifies $\hat{B}\in \mathcal{N}%
^{-1}(X,\hat{Y}).$ A coincidence for the pair $(\hat{A},\hat{B})$ is also a
coincidence for $(A,B).$

It suffices, thus, to show that $A\cap B\neq \emptyset $ under hypothesis $%
(i),$ i.e., when $Y$ is compact. Since $\forall y\in Y,B^{-1}(y)\neq
\emptyset ,$ it follows that $\{B(x)\}_{x\in X}$ forms an open cover of $Y.$
Similarly, $\{A^{-1}(y)\}_{y\in Y}$ is an open cover of $X.$ Let $%
\{B(x_{i})\}_{i=1}^{n}$ be a finite subcover of $Y,$ let $%
D:=conv(\{x_{1},\ldots ,x_{n}\}),$ a convex compact subset of $X,$ and let $%
\{A^{-1}(y_{j})\}_{j=1}^{m}$ be an open subcover of $D.$ Consider the convex
compact subset $M:=conv(\{y_{1},\ldots ,y_{m}\})$ of $Y.$ $M$ can be covered
by a subfamily of $\{B(x_{i})\}_{i=1}^{n},$ which, for simplicity we also
denote $\{B(x_{i})\}_{i=1}^{n}.$ We can assume with no loss of generality
that $\{B(x_{i})\}_{i=1}^{n}$ and $\{A^{-1}(y_{j})\}_{j=1}^{m}$ are minimal
covers of $M$ and $D$ respectively. That is, for any $k\in \{1,\ldots
,n\},M\nsubseteq \bigcup\nolimits_{i=1,i\neq k}^{n}B(x_{i}),$ and for any $%
l\in \{1,\ldots ,m\},D\nsubseteq \bigcup\nolimits_{j=1,j\neq
l}^{m}A^{-1}(y_{j}).$ Consider the compact convex sets $M_{i}=M\setminus
B(x_{i})$ for for $i=1,\ldots ,n,$ and $D_{j}=D\setminus A^{-1}(y_{j})$ for $%
j=1,\ldots ,m.$ The fact that $M\subseteq \bigcup\nolimits_{i=1}^{n}B(x_{i})$
is equivalent to $\bigcap\nolimits_{i=1}^{n}M_{i}=\emptyset $, and $%
D\subseteq \bigcup\nolimits_{j=1}^{m}A^{-1}(y_{j})$ is equivalent to $%
\bigcap\nolimits_{j=1}^{m}D_{j}=\emptyset .$ The minimality of the covers $%
\{B(x_{i})\}_{i=1}^{n}$ and $\{A^{-1}(y_{j})\}_{j=1}^{m}$ amounts to $M\cap
\bigcap\nolimits_{i=1,i\neq k}^{n}M_{i}\neq \emptyset $ for any $k\in
\{1,\ldots ,n\},$ and $D\cap \bigcap\nolimits_{j=1,j\neq l}^{m}D_{j}\neq
\emptyset $ for any $l\in \{1,\ldots ,m\}.$ All conditions of Klee's theorem
(see Remark 4) are thus satisfied for both families of compact convex sets $%
\{M,M_{1},\ldots ,M_{n}\}$ and $\{D,D_{1},\ldots ,D_{m}\}.$ Hence, $%
M\nsubseteq \bigcup\nolimits_{i=1}^{n}M_{i}$ and $D\nsubseteq
\bigcup\nolimits_{j=1}^{m}D_{j}.$ Let \thinspace $y_{0}\in M$ but $%
y_{0}\notin M_{i}$ for all $i=1,\ldots ,n,$ and let $x_{0}\in D$ but $%
x_{0}\notin D_{j}$ for all $j=1,\ldots ,m.$ Clearly, $y_{0}\in
B(x_{i})\Longleftrightarrow x_{i}\in B^{-1}(y_{0})$ for all $i=1,\ldots ,n,$
and $x_{0}\in A^{-1}(y_{j})\Longleftrightarrow y_{j}\in A(x_{0})$ for all
for all $j=1,\ldots ,m.$ The sections $B^{-1}(y_{0})$ and $A(x_{0})$ being
convex sets, it follows that $x_{0}\in D=conv(\{x_{1},\ldots
,x_{n}\})\subset B^{-1}(y_{0})$ and $y_{0}\in M=conv(\{y_{1},\ldots
,y_{m}\})\subset A(x_{0}).$ The proof is finished as $(x_{0},y_{0})\in A\cap
B.$
\end{proof}

\section{Analytic Formulations and Applications}

This section illustrates how the geometric results in the preceding section,
Theorems 6, 9, and 12, are key in deriving a number of landmark results in
functional analysis. Intersection theorems as well as fixed point and
coincidence theorems have analytical formulations as solvability theorems
for systems of nonlinear inequalities (see [1, 2, 3, 4, 9]). These
analytical formulations are often more practical when it comes to
applications. We start with the analytical formulation of the convex KKM\
principle (equivalently, the fixed point theorem for von Neumann relations)
and we derive from it, in a simple and straightforward way, two fundamental
results.\bigskip

\subsection{Alternatives for Systems of Nonlinear Inequalities and
Applications}

Theorems 6 and 9 can be expressed in terms of an alternative for nonlinear
systems of inequalities \textit{\`{a} la Ky Fan}.

Recall first the basic concepts of semicontinuity and quasiconvexity for
real functions.\bigskip

\begin{definition}
\textit{A real function }$f:X\longrightarrow 
\mathbb{R}
$\textit{\ defined on a subset }$X$\textit{\ of a t.v.s is:}

(i) \textit{quasiconvex} if $\forall \lambda \in 
\mathbb{R}
,$ the level set $\{x\in X;f(x)\leq \lambda \}$ is a convex subset of $X;$

(ii) \textit{quasiconcave} if $-f$ is quasiconvex;

(iii) \textit{lower semicontinuous} \textit{(l.s.c.) }if $\forall \lambda
\in 
\mathbb{R}
,$ the level set $\{x\in X;f(x)\leq \lambda \}$\textit{\ }is a closed subset
of $X;$

(iv) \textit{upper semicontinuous} \textit{(u.s.c.) }if $-f$ is
l.s.c..\bigskip
\end{definition}

Naturally, every convex functional is quasiconvex and the converse is false.
Also, a real function on a topological space is continuous if and only if it
is both upper and lower semicontinuous.

\begin{theorem}
Let $X$ be a convex subset of a t.v.s. $E$, $Y$ a non-empty subset of $X,$
and $f:X\times Y\longrightarrow 
\mathbb{R}
$\textit{\ a function satisfying:}

(i) $x\mapsto f(x,y)$\ is l.s.c. and quasiconvex on $X,$ for each fixed $%
y\in Y.$

(ii) $y\mapsto f(x,y)$\ quasiconcave on $Y,$\ for each fixed $x\in X;$

Assume that for a given\ $\lambda \in 
\mathbb{R}
,$ there exist a compact subset $K$ of $X$ and a convex compact subset $D$
of $Y$

such that $\forall x\in X\setminus K,\exists y\in D$ with $f(x,y)>\lambda .$

Then the following alternative holds:

(A) there exists $x_{0}\in X$\ such that $f(x_{0},x_{0})>\lambda ,$\ or

(B) there exists $\bar{x}\in Y$\ such that $f(\bar{x},y)\leq \lambda ,$\ for
all $y\in Y.$

Consequently, when $\lambda =\sup_{x\in X}f(x,x),$ (A) is impossible and 
\begin{equation*}
\inf_{x\in X}\sup_{y\in Y}f(x,y)\leq \sup_{x\in X}f(x,x).
\end{equation*}
\end{theorem}

\begin{proof}
Let $A(x):=\{y\in Y:f(x,y)>\lambda \},x\in X.$ All hypotheses of Theorem 9
are satisfied except, possibly, the non-emptiness of the sections $A(x).$
Thus, either $A(x)\neq \emptyset ,\forall x\in X,$ hence $A$ is a von
Neumann relation, and therefore has a fixed point ((A) holds), or $A(\bar{x}%
)=\emptyset $ for some $\bar{x}\in X,$ i.e., $A$ is not a von Neumann
relation and (B) is satisfied.
\end{proof}

This is a particular instance of the celebrated I\textit{nfsup Inequality} 
\textit{of Ky Fan} with a weaker compactness condition.

Landmark theorems of nonlinear functional analysis follow immediately from
Theorem 14; therefore, indirectly, from the separation of closed convex sets
in finite dimension (Proposition 2). We refer to H. Br\'{e}zis [6] for an
account and applications of the next two fundamental results.

\begin{corollary}
(Mazur-Schauder Theorem) Let $X$ be a non-empty closed convex subset of a
reflexive Banach space $E$ and let $\varphi :X\longrightarrow 
\mathbb{R}
$ be a lower semicontinuous, quasiconvex and coercive (i.e. $%
\lim_{||x||\rightarrow \infty }\varphi (x)=\infty )$ functional. Then $%
\varphi $ achieves its minimum on $X.$
\end{corollary}

\begin{proof}
Let $\lambda =0,$ $Y=X,$ and $f(x,y)=\varphi (x)-\varphi (y)$ in Theorem 14.
Let $K$ be the intersection of $X$ with a closed ball with radius $M>0$
centered at the origin of $E$ and such that if $x\in X$ with $\Vert x\Vert
>M $ then $\varphi (x)>\varphi (y)$ for some $y\in K.$ Such a non-empty set $%
K$ exists due to the coercivity of $\varphi .$ Since $E$ is reflexive, $K$
is weakly compact. One readily verifies that the hypotheses of Theorem 14
with $X,Y,f,K,D=K,$ and $\lambda =0$ all hold: $f$ is l.s.c. and quasiconvex
in $x, $ and quasiconcave in $y.$ Clearly, possibility (A) of Theorem 14
cannot hold. Hence (B) is true: there exists $\bar{x}\in X$\ such that $f(%
\bar{x},y)=\varphi (\bar{x})-\varphi (y)\leq 0,$\ for all $y\in X.$

We now derive from the nonlinear alternative in Theorem 14 the celebrated
theorem of Stampacchia for variational inequalities. Recall that given a
normed space $E,$ a form $a:E\times E\longrightarrow 
\mathbb{R}
$ is said to be (i) \textit{bilinear} if it linear in each of its arguments;
(ii) \textit{continuous} if there exists a constant $C>0$ with $|a(x,y)|\leq
C||x||||y||$ for all $x,y\in E;$ and (iii) \textit{coercive} if there exists
a constant $\alpha >0$ with $a(x,x)\geq \alpha ||x||^{2}$ for all $x\in E.$
\end{proof}

\begin{corollary}
(Stampacchia Theorem) Let $E$ be a reflexive Banach space, $a:E\times
E\longrightarrow 
\mathbb{R}
$ be a continuous and coercive bilinear form, and let $\ell
:E\longrightarrow 
\mathbb{R}
$ be a bounded linear functional. Given a non-empty closed and convex subset 
$X$ in $E,$ there exists a unique $\bar{x}\in X$\ such that $a(\bar{x},\bar{x%
}-y)\leq \ell (\bar{x})-\ell (y)$ \ for all $y\in X.$
\end{corollary}

\begin{proof}
For the existence, we apply Theorem 14 to $f:X\times X\longrightarrow 
\mathbb{R}
$ defined by $f(x,y):=a(x,x-y)-\ell (x-y),(x,y)\in X\times X,\lambda =0,$ $%
D=\{y_{0}\}$ with $0\neq y_{0}\in X$ arbitrary, and $K:=\{x\in X:\Vert
x\Vert \leq M\}$ where 
\begin{equation*}
M:=\frac{1}{2}(\beta +\sqrt{\beta ^{2}+4\gamma }),
\end{equation*}%
$\beta =(C||y_{0}||+||\ell ||)/\alpha $ and $\gamma =||\ell
||||y_{0}||/\alpha .$

Indeed, first note that if $E$ is equipped with the weak topology, then $%
f(x,y)$ is $l.s.c.$ and quasiconvex in $x$ and quasiconcave in $y$ (it is in
fact linear and continuous for the norm topology in both arguments). Since $%
X $ is closed and convex, it follows that $K$ is a closed, convex and
bounded, hence weakly compact, subset of $X.$ $D$ is obviously a weakly
compact subset of $X.$ Note now that if $f(x,y_{0})\leq 0$ for any given $%
x\in X,$ i.e., $a(x,x)\leq a(x,y_{0})+\ell (x-y_{0}),$ then $\Vert x\Vert $
satisfies a quadratic inequality and is bounded above by $M$:%
\begin{equation*}
\left. 
\begin{array}{c}
\\ 
\Longrightarrow \\ 
\Longleftrightarrow \\ 
\Longrightarrow%
\end{array}%
\right. \left. 
\begin{array}{l}
\alpha ||x||^{2}\leq a(x,x)\leq C\Vert x\Vert \Vert y_{0}\Vert +\Vert \ell
\Vert \Vert x\Vert +\Vert \ell \Vert \Vert y_{0}\Vert \\ 
\alpha ||x||^{2}-(C\Vert y_{0}\Vert +\Vert \ell \Vert )\Vert x\Vert -\Vert
\ell \Vert \Vert y_{0}\Vert \leq 0 \\ 
||x||^{2}-\beta \Vert x\Vert -\gamma \leq 0 \\ 
\Vert x\Vert \leq \frac{1}{2}(\beta +\sqrt{\beta ^{2}+4\gamma })=M.%
\end{array}%
\right.
\end{equation*}%
Consequently, if $x\in X,\Vert x\Vert >M,$ then $f(x,y_{0})>0$ and the
compactness condition in Theorem 14 is satisfied. Since $f(x,x)=0$ for any $%
x\in X,$ (A) of Theorem 14 is impossible, and (B)\ holds, i.e., $f(\bar{x}%
,y)=a(\bar{x},\bar{x}-y)-\ell (\bar{x})+\ell (y)\leq 0$ for some $\bar{x}\in
X$ and all $y\in X$ and the proof of the existence is complete.

The uniqueness follows at once from the bilinearity and the coercivity of
the form $a$ as follows: if $a(\bar{x}_{i},\bar{x}_{i}-y)-\ell (\bar{x}%
_{i})+\ell (y)\leq 0$ for two elements $\bar{x}_{i}\in X,i=1,2,$ and all $%
y\in X,$ then adding $a(\bar{x}_{1},\bar{x}_{1}-\bar{x}_{2})\leq \ell (\bar{x%
}_{1})-\ell (\bar{x}_{2})$ to $a(\bar{x}_{2},\bar{x}_{2}-\bar{x}_{1})\leq
\ell (\bar{x}_{2})-\ell (\bar{x}_{1})$ gives $0\leq \alpha \Vert \bar{x}_{1}-%
\bar{x}_{2}\Vert ^{2}\leq a(\bar{x}_{1}-\bar{x}_{2},\bar{x}_{1}-\bar{x}%
_{2})\leq 0,$ i.e., $\bar{x}_{1}=\bar{x}_{2}.$
\end{proof}

The coincidence $(\mathcal{N},\mathcal{N}^{-1})$ (Theorem 12) can be
expressed in analytical terms as a second alternative for nonlinear systems
of inequalities as follows:

\begin{theorem}
\textit{Let }$X$\textit{\ and }$Y$\textit{\ be two convex subsets of
topological vector spaces and let }$f,g:X\times Y\longrightarrow 
\mathbb{R}
$\textit{\ be two functions satisfying:}

\textit{(i) }$f(x,y)\leq g(x,y)$\textit{\ for all }$(x,y)\in X\times Y;$

\textit{(ii) }$x\mapsto f(x,y)$\textit{\ is quasiconcave on }$X,$\textit{\
for each fixed }$y\in Y;$

\textit{(iii) }$y\mapsto f(x,y)$\textit{\ is l.s.c. and quasiconvex on }$Y,$%
\textit{\ for each fixed }$x\in X;$

\textit{(iv) }$x\mapsto g(x,y)$\textit{\ is u.s.c. and quasiconcave on }$X,$%
\textit{\ for each fixed }$x\in X;$

\textit{(v) }$y\mapsto g(x,y)$\textit{\ is quasiconvex on }$Y,$\textit{\ for
each fixed }$x\in X.$

(vi) \textit{Given }$\lambda \in 
\mathbb{R}
$ arbitrary, assume that \textit{either }$Y$\textit{\ is compact, or }$X$ is
compact\textit{, or there exist a compact subset }$K$ of $X$ and a convex
compact subset $C$ of $Y$ such that for any $x\in X\setminus K$ there exists 
$y\in C$ with $g(x,y)<\lambda .$

\textit{Then one of the following statements holds:}

(\textit{A) there exists }$\bar{x}\in X$\textit{\ such that }$g(\bar{x}%
,y)\geq \lambda ,$\textit{\ for all }$y\in Y;$ \textit{or}

(\textit{B) there exists }$\bar{y}\in Y$\textit{\ such that }$f(x,\bar{y}%
)\leq \lambda ,$\textit{\ for all }$x\in X.$
\end{theorem}

\begin{proof}
Simply apply Theorem 12 to $A,B\subset X\times Y$ defined as:%
\begin{equation*}
A:=\{(x,y):g(x,y)<\lambda \}\text{ and }B:=\{(x,y):f(x,y)>\lambda \}\text{.}
\end{equation*}

Note that in view of (i) a coincidence between $A$ and $B$ is impossible as
it yields $\lambda <\lambda .$ Since all hypotheses of Theorem 12 are
satisfied save for $A(x)\neq \emptyset $ for all $x\in X$ and $B^{-1}(y)\neq
\emptyset $ for all $y\in Y,$ it follows that either $A(\bar{x})=\emptyset $
for some $\bar{x}\in X$ (thesis (A)) or $B^{-1}(\bar{y})=\emptyset $ for
some $\bar{y}\in Y$ (thesis (B)).
\end{proof}

\begin{remark}
Theorem 17 implies $\alpha =\sup_{X}\inf_{Y}g(x,y)\geq
\inf_{Y}\sup_{X}f(x,y)=\beta .$

Indeed, assuming that $\alpha <\beta <\infty ,$ let $\lambda $ be an
arbitrary but fixed real number strictly between $\alpha $ and $\beta .$ By
Theorem 17, either there exists $\bar{y}\in Y$\ such that $f(x,\bar{y})\leq
\lambda ,$\ for all $x\in X$ thus $\beta \leq \lambda <\beta $ which is
impossible, or there exists $\bar{x}\in X$\ such that $g(\bar{x},y)\geq
\lambda ,$\ for all $y\in Y$ thus $\alpha \geq \lambda >\alpha $ which is
absurd. Hence $\alpha \geq \beta .$
\end{remark}

Maurice Sion's formulation of the von Neumann Minimax Theorem follows
immediately with $f=g$ (see [1]):

\begin{corollary}
(Sion-von Neumann Minimax Theorem) Let $X$ and $Y$ be convex subsets of
topological vector spaces and let $f$ be a real function on $X\times Y$ such
that:

(i) $x\mapsto f(x,y)$ is quasiconcave and u.s.c. on $X\,$\ for each fixed $%
y\in Y;$

(ii) $y\mapsto f(x,y)$ is quasiconvex and l.s.c. on $Y\,$\ for each fixed $%
x\in X;$

Assume that either $X$ is compact or $Y$ is compact. Then:%
\begin{equation*}
\alpha =\sup_{X}\inf_{Y}f(x,y)=\inf_{Y}\sup_{X}f(x,y)=\beta
\end{equation*}
\end{corollary}

\begin{proof}
The inequality $\alpha \leq \beta $ is always true and $\alpha \geq \beta $
follows from Remark 18.
\end{proof}

\begin{remark}
If both $X$ and $Y$ are compact, the infsup equality in Corollary 19 is a
minmax equality and is equivalent to the existence of a saddle point $%
(x_{0},y_{0})$ for the function $f(x,y),$ i.e., $f(x,y_{0})\leq
f(x_{0},y),\forall (x,y)\in X\times Y.$
\end{remark}

We end this section with a short proof of the Markov-Kakutani fixed point
theorem for abelian families of continuous affine mappings in linear
topological spaces having separating duals\footnote{%
A t.v.s. $E$ \textit{has separating dual} if for each $x\in E,x\neq 0,$
there exists a bounded linear form $\ell \in E^{\prime },$ the topological
dual of $E,$ such that $\ell (x)\neq 0.$ Locally convex topological vector
spaces have separating duals. Sequence spaces $\ell ^{p},0<p<1,$ and Hardy
spaces $H^{p},0<p<1,$ are instances of non-locally convex spaces with
separating duals.}.

Recall that a mapping $\phi $ from a convex set $X$ into a vector space is
said to be \textit{affine} if and only if $\phi (\sum_{i=1}^{n}\lambda
_{i}x_{i})=\sum_{i=1}^{n}\lambda _{i}\phi (x_{i})$ for any convex
combination $\sum_{i=1}^{n}\lambda _{i}x_{i},\lambda _{i}\geq
0,\sum_{i=1}^{n}\lambda _{i}=1,$ in $X.$ The key ingredient is the following
fixed point property for continuous affine transformations of a compact
convex set.

\begin{corollary}
Let $X$ be a non-empty compact convex subset of a t.v.s. $E$ with separating
dual $E^{\prime }$ and let $\phi :X\longrightarrow X$ be a continuous affine
mapping. Then $\phi $ has a fixed point.
\end{corollary}

\begin{proof}
The proof is a simplification of the treatment in [9]. Define $f:X\times
E^{\prime }\longrightarrow 
\mathbb{R}
$ by $f(x,\ell )=\ell (\phi (x)-x),(x,\ell )\in X\times E^{\prime }.$ It
suffices to prove the existence of $x_{0}\in X$ such that $f(x_{0},\ell
)\leq 0,\forall \ell \in E^{\prime },$ for this would imply $\ell (\phi
(x_{0})-x_{0})=0,\forall \ell \in E^{\prime },$ i.e., $\phi (x_{0})-x_{0}=0$
and the proof is complete.

This amounts to showing that $\bigcap_{\ell \in E^{\prime }}A(\ell )\neq
\emptyset $ for the relation $A:=\{(\ell ,x)\in E^{\prime }\times X:f(x,\ell
)\leq 0\}.$

Since for each fixed $\ell \in E^{\prime },$ the function $f(x,\ell )$ is
l.s.c. in $x,$ then each $A(\ell )$ is a closed, hence compact, subset of $%
X. $ It suffices, therefore, to show that the collection $\{A(\ell ):\ell
\in E^{\prime }\}$ has the finite intersection property. Consider, to this
aim, a finite collection of bounded linear functionals $L:=\{\ell
_{1},\ldots ,\ell _{n}\}\subset E^{\prime },$ and let $Y=conv(L),$ a convex
compact subset of $E^{\prime }.$ The restriction of $f(x,\ell )$ to $X\times
Y$ is obviously u.s.c. and quasiconcave in $x$ and l.s.c. and quasiconvex in 
$\ell .$ Since both $X$ and $Y$ are compact and convex, it follows from
Remark 20 that there exists $(x_{0},\ell _{0})\in X\times Y$ with $%
f(x_{0},\ell )\leq f(x,\ell _{0})$ for all $(x,\ell )\in X\times Y,$ i.e., $%
\ell (\phi (x_{0})-x_{0})\leq \ell _{0}(\phi (x)-x)$ for all $(x,\ell )\in
X\times Y.$ Let $\hat{x}$ be such that $\ell _{0}(\hat{x})=\max_{x\in X}\ell
_{0}(x).$ Since $\phi (\hat{x})\in X,$ it follows that $\ell _{0}(\phi (\hat{%
x})-\hat{x})\leq 0$ and, consequently, $\ell (\phi (x_{0})-x_{0})\leq 0,$
for all $\ell \in Y,$ in particular, $\ell _{i}(\phi
(x_{0})-x_{0})=f(x_{0},\ell _{i})\leq 0$ for all $\ell _{i}\in L,$ and the
proof is complete.
\end{proof}

The Markov-Kakutani follows by a standard compactness argument. Recall that
a family $\mathcal{\tciFourier }=\{\phi \}$ of mappings is said to be 
\textit{abelian} if $\phi _{1}\phi _{2}=\phi _{2}\phi _{1}$ for all $\phi
_{1},\phi _{2}\in \tciFourier .$

\begin{corollary}
(Theorem of Markov-Kakutani) Let $X$ be a non-empty compact convex subset of
a t.v.s. $E$ with separating dual $E^{\prime }$ and let $\tciFourier $ be an
abelian family of continuous affine transformations from $X$ into itself$.$
Then, there exists $x_{0}\in X$ such that $\phi (x_{0})=x_{0}$ for every $%
\phi \in \tciFourier .$
\end{corollary}

\begin{proof}
For any given $\phi \in \tciFourier ,$ let $Fix(\phi )$ be the set of its
fixed points. We show that $\bigcap_{\phi \in \tciFourier }Fix(\phi )\neq
\emptyset .$ Clearly, for each $\phi \in \tciFourier ,$ $Fix(\phi )$ is
non-empty (by Corollary 21), convex (as $\phi $ is affine), and closed hence
compact in $X.$ It suffices to show that the family $\{Fix(\phi ):$ $\phi
\in \tciFourier \}$ has the finite intersection property, i.e., $%
\bigcap_{i=1}^{n}Fix(\phi _{i})\neq \emptyset $ for any $\{\phi _{1},\ldots
,\phi _{n}\}$ $\subset \tciFourier $. The proof is by induction on $n.$ For $%
n=1,$ clearly $Fix(\phi _{1})\neq \emptyset $ (Corollary 21). Assume that
the statement is true for any family $\{\phi _{1},\ldots ,\phi _{k}\}$ $%
\subset \tciFourier $ with $k=n-1$ and let $\{\phi _{1},\ldots ,\phi _{n}\}$ 
$\subset \tciFourier $ be arbitrary. For any $x\in
\bigcap_{i=1}^{n-1}Fix(\phi _{i}),$ $\phi _{n}(x)=\phi _{n}(\phi
_{i}(x))=\phi _{i}(\phi _{n}(x))$ for all $i=1,\ldots ,n-1,$ i.e., $\phi
_{n}(x)\in \bigcap_{i=1}^{n-1}Fix(\phi _{i}).$ Thus $\phi _{n}$ maps the
non-empty compact convex set $\bigcap_{i=1}^{n-1}Fix(\phi _{i})$ into
itself. By Corollary 21 again, it has a fixed point $\bar{x}=\phi _{n}(\bar{x%
})\cap \bigcap_{i=1}^{n-1}Fix(\phi _{i}),$ i.e., $\bar{x}\in
\bigcap_{i=1}^{n}Fix(\phi _{i}).$
\end{proof}

\section{Concluding Remarks}

It is well established that the Markov-Kakutani fixed point theorem implies
the Hahn-Banach theorem (Kakutani [16]). The two results are indeed
equivalent (for a short and elegant proof of the converse, see D. Werner
[21]). Since we derived here the convex KKM\ theorem from the theorem on the
separation of convex sets, we have thus established the equivalence of the
Hahn-Banach theorem, Klee's intersection theorem, the convex KKM theorem,
the fixed point theorem for von Neumann relations, the Sion-von Neumann
minimax theorem, and the Markov-Kakutani fixed point theorem.

Although the convex KKM theorem is a particular instance of the KKM\
principle of Ky Fan, and since the fixed point and coincidence properties in
Theorems 9 and 12 are special cases of similar results for so-called $F$ and 
$F^{\ast }$ maps (see [2, 3, 4]), the interest here resides in the use of
simple arguments of convexity rather than the Brouwer fixed point theorem or
Sperner's lemma. It would be most interesting to know if the Ky Fan KKM\
principle (or any of its equivalent results) can be derived directly from
the convex KKM theorem. In other words, can any of the question marks below
be settled? (The smaller arrows are established; FPT stands for Fixed Point
Theorem.)

\begin{equation*}
\begin{bmatrix}
\begin{array}{c}
\text{Convex Klee} \\ 
\text{(Proposition 3)}%
\end{array}
& \rightleftarrows & 
\begin{array}{c}
\text{Convex KKM} \\ 
\text{(Theorem 6)}%
\end{array}
& \rightleftarrows & 
\begin{array}{c}
\text{FPT for }\mathcal{N} \\ 
\text{(Theorem 9)}%
\end{array}
\\ 
\uparrow {\Huge \downarrow ?} & \nwarrow {\Huge \searrow ?} & \uparrow 
{\Huge \downarrow ?} & {\Huge ?\swarrow }\nearrow & \uparrow {\Huge %
\downarrow ?} \\ 
\text{KKM\ Lemma} & \rightleftarrows & \text{Ky Fan KKM} & \rightleftarrows
& \text{Browder-Ky Fan FPT}%
\end{bmatrix}%
\end{equation*}

\bigskip

Let us point out, for the reader's benefit that the following equivalences
have been established in [15]:%
\begin{equation*}
\begin{bmatrix}
\text{Brouwer FPT} &  & \rightleftarrows &  & \text{Topological Klee} \\ 
& \nwarrow \searrow &  & \swarrow \nearrow &  \\ 
&  & \text{Ky Fan KKM} &  & 
\end{bmatrix}%
\end{equation*}

Here, the "topological Klee theorem" [15] reads:\bigskip

\textit{A family of }$n$\textit{\ closed convex sets in a topological vector
space has a non-empty intersection if and only if the union of the }$n$%
\textit{\ sets is }$(n-2)$\textit{-connected and the intersection of every }$%
n-1$\textit{\ of them is non-empty.}\bigskip

This topological version of Klee's theorem yields the equivalent formulation
of the Brouwer fixed point theorem:

\begin{center}
\textit{The }$n$\textit{-sphere }$S^{n}$\textit{\ is not }$n$\textit{%
-connected.}
\end{center}

Indeed, the $n-$dimensional faces of the $(n+1)-$simplex $\Delta ^{n+1}$
form a family of $n+2$ closed convex sets in $%
\mathbb{R}
^{n+2}$. Moreover, every intersection of $n+1$ of them is non-empty, but the
whole intersection is empty. Hence, their union - which consists of the
boundary $\partial \Delta ^{n+1}$ - is not $n-$connected. Since $\partial
\Delta ^{n+1}$ is homeomorphic to $S^{n}$, it follows that $S^{n}$ is not $%
n- $connected. This establishes the implication:\ topological Klee theorem $%
\Longrightarrow $ Brouwer FPT. The equivalence Brouwer FPT $%
\Longleftrightarrow $ Ky Fan KKM\ principle is well-established (see e.g.,
[9]).

Also, since every convex set in a topological vector space is contractible,
hence $n-$connected for every $n\geq 0$, the topological Klee theorem
implies Proposition 3. Does the converse hold true?

\end{document}